\documentclass{amsart}

\usepackage{amssymb}
\usepackage{epsfig}
\usepackage{amsthm}
\usepackage{amsfonts,amsmath,bbm,mathrsfs}
\usepackage[utf8]{inputenc}
\usepackage{hyperref}

\newcommand{\N}{\mathbb{N}}

\newcommand{\R}{\mathbb{R}}
\newcommand{\C}{\mathbb{C}}

\newtheorem{theorem}{Theorem}[section]
\newtheorem{lemma}[theorem]{Lemma}
\newtheorem{corollary}[theorem]{Corollary}

\theoremstyle{remark}
\newtheorem{remark}[theorem]{Remark}
\newtheorem{example}[theorem]{Example}

\renewcommand{\Re}{\mathrm{Re}\,}	% Human readable Re
\renewcommand{\Im}{\mathrm{Im}\,}	% and Im

\newcommand{\rmi}{\mathrm{i}}
\newcommand{\rme}{\mathrm{e}}
\newcommand{\xhat}{\hat{x}}
\newcommand{\Ol}{{\mathcal O}}
\newcommand{\loc}{{\mathrm{loc}}}

\newcommand{\ds}{\,\mathrm{d}S}
\newcommand{\dy}{\,\mathrm{d}y}
\newcommand{\dx}{\,\mathrm{d}x}

\newcommand{\Acal}{\mathcal{A}}
\newcommand{\Mcal}{\mathcal{M}}

\DeclareMathOperator{\supp}{supp}

\begin{document}

\title[Analyticity of far field patterns]{A note on analyticity properties of far field patterns}

\author{Roland Griesmaier}
\address{Mathematisches Insitut, Universit\"at Leipzig, 04009 Leipzig, Germany}
\email{griesmaier@math.uni-leipzig.de}
\urladdr{http://www.math.uni-leipzig.de/$\sim$griesmaier/}

\author{Nuutti Hyv\"onen}
\address{Aalto University, Department of Mathematics and Systems Analysis, FI-00076 Aalto, Finland}
\email{nuutti.hyvonen@aalto.fi}
\urladdr{http://users.tkk.fi/nhyvonen/}

\author{Otto Seiskari}
\address{Aalto University, Department of Mathematics and Systems Analysis, FI-00076 Aalto, Finland}
\email{otto.seiskari@aalto.fi}

\thanks{This work was supported in part by the Academy of Finland (project 135979)}

\subjclass[2000]{35R30, 35Q60}
\keywords{Far field pattern, analytic functions, inverse scattering, partial data}

\hypersetup{%
  pdftitle={\shorttitle},%
  pdfauthor={},%
  pdfsubject={},%
  pdfkeywords={},%
  pdfborder={0 0 0},%
  colorlinks={false}%
}

\begin{abstract}
  In scattering theory the far field pattern describes the directional dependence of a time-harmonic wave scattered by an obstacle or inhomogeneous medium, when observed sufficiently far away from these objects.
  Considering plane wave excitations, the far field pattern can be written as a function of two variables, namely the direction of propagation of the incident plane wave and the observation direction, and it is well-known to be separately real analytic with respect to each of them.
  We show that the far field pattern is in fact a jointly real analytic function of these two variables.
\end{abstract}

\maketitle

  \noindent
{\footnotesize \textbf{Last modified:} \today}

\section{Introduction}
\label{sec:intro}
One of the main themes of inverse scattering theory for time-harmonic acoustic and electromagnetic waves is to recover information on obstacles or inhomogeneous media from knowledge of far field patterns of scattered fields excited by plane wave incident fields.
Sufficiently far away from the scatterers every scattered field has the asymptotic behavior of an outgoing spherical wave, and the far field pattern describes its directional dependence.
Taking into account the incident field, the far field pattern $u^\infty(\xhat;\theta)$, $\xhat, \theta \in S^{d-1}$, where $S^{d-1}$ denotes the unit sphere in $\R^d$, $d=2,3$, is a function of two variables --- the observation direction $\xhat$ and the direction of propagation $\theta$ of the incident plane wave.
This function is well known to be separately real analytic, i.e., for each $\theta \in S^{d-1}$ the function $u_\infty(\,\cdot\,;\theta)$ is real analytic on $S^{d-1}$, while for each $\xhat \in S^{d-1}$ the function $u_\infty(\xhat;\,\cdot\,)$ is real analytic on $S^{d-1}$ (see, e.g., Colton and Kress \cite{ColKre83}).
In this work we establish that the far field pattern is in fact jointly real analytic on $S^{d-1}\times S^{d-1}$.

To this end, we consider the special case of scattering from an inhomogeneous medium (see, e.g., Colton and Kress~\cite{Colton92}), but our approach immediately extends to obstacle scattering problems as well.
The main idea is to utilize a factorization of the so-called far field operator, which maps superpositions of plane wave incident fields to the corresponding far field patterns, to show that the far field pattern can locally be extended to a separately holomorphic function.
Then the assertion follows from Hartogs' theorem \cite{Har06}, which is a fundamental result in the theory of several complex variables. A similar approach has previously been employed in a related two-dimensional setting of electrical impedance tomography \cite{Hyvonen12}.

The joint real analyticity of the far field pattern has an interesting consequence for the classical uniqueness results for inverse medium scattering problems due to Bukhgeim~\cite{Buk08}, Nachman~\cite{Nac88}, Novikov~\cite{Nov88}, and Ramm~\cite{Ram88}, which state that (under appropriate regularity assumptions) the index of refraction of a compactly supported inhomogeneous object is uniquely determined by the knowledge of the far field pattern $u_\infty(\xhat;\theta)$ for all $(\xhat,\theta) \in {S^{d-1}\times S^{d-1}}$.
The separate real analyticity implies that it is actually sufficient to know the far field pattern on some open subset of $S^{d-1}\times S^{d-1}$ (see, e.g., \cite{Colton92}).
From the joint real analyticity shown in this work it follows that the far field pattern on $S^{d-1}\times S^{d-1}$, and thus the index of refraction, is completely determined by only knowing $u_\infty$ and all of its derivatives at a single point $(x,\theta) \in S^{d-1}\times S^{d-1}$.

The outline of this article is as follows.
In the next section we introduce our mathematical setting and state the main result, which is then proved in Section~\ref{sec:proof}.

\section{Mathematical setting and main result}
\label{sec:main}
We consider scattering of time harmonic acoustic waves by an inhomogeneous medium in $\R^d$, $d=2,3$, caused by a plane wave \emph{incident field} $u^i(x;\theta) := \rme^{\rmi k x\cdot\theta}$, $x\in\R^d$, with the \emph{incident direction} $\theta\in S^{d-1}$, satisfying the homogeneous Helmholtz equation
\begin{equation*}
  \Delta u^i(\,\cdot\,;\theta) + k^2 u^i(\,\cdot\,;\theta) \,=\, 0 \qquad  \mbox{in $\R^d$} \,,
\end{equation*}
where $k>0$ is a (fixed) \emph{wave number}.
The scattering properties of the medium are characterized by its \emph{refractive index} $n \in L^\infty(\R^d)$, and we assume that ${\Re(n) \geq 0}$ and $\Im(n) \geq 0$ almost everywhere in $\R^d$ as well as that $\supp (n-1) \subset \R^d$ is compact.
Then the direct scattering problem is to determine the \emph{total field} $u(\,\cdot\,;\theta) = u^i(\,\cdot\,;\theta) + u^s(\,\cdot\,;\theta) \in H^1_{\rm loc}(\R^d)$ satisfying
\begin{equation}
  \label{eq:forward}
  \Delta u(\,\cdot\,;\theta) + k^2 n u(\,\cdot\,;\theta) \,=\, 0 \qquad  \mbox{in $\R^d$} \,, 
\end{equation}
such that the \emph{scattered field} $u^s(\,\cdot\,;\theta)$ satisfies the \emph{Sommerfeld radiation condition}
\begin{equation}
  \label{eq:Sommerfeld}
  \lim_{r \to \infty} r^{(d-1)/2} \Big( \frac{\partial u^s(\,\cdot\,;\theta)}{\partial r} - {\rm i}k u^s(\,\cdot\,;\theta) \Big) \,=\, 0 \,,
\end{equation}
where $r := |x|$ and \eqref{eq:Sommerfeld} holds uniformly with respect to $\xhat := x/r \in S^{d-1}$. 
It is well known that \eqref{eq:forward}--\eqref{eq:Sommerfeld} is uniquely solvable (cf., e.g., \cite[Sec.~6.2]{Kir11}).
In fact, the unique weak solution belongs to $H^2_{\rm loc}(\R^d)$ and is smooth away from $\supp(n-1)$ due to standard interior regularity results for elliptic partial differential equations (see, e.g., \cite[p.~125, Thm.~3.2]{Lions72}).

The scattered field has the asymptotic behavior
\begin{equation*}
  u^s(x;\theta) = \frac{\rme^{\rmi k |x|}}{|x|^{(d-1)/2}} \biggl( u_\infty(\xhat; \theta) + \Ol\biggl(\frac1{|x|}\biggr) \biggr)
\end{equation*}
uniformly with respect to $\xhat \in S^{d-1}$, where the \emph{far field pattern} $u_\infty(\,\cdot\,;\theta)$ is given by
\begin{equation}
  \label{eq:farfieldpattern}
  u_\infty(\xhat;\theta) = c_d k^2 \int_{\R^d} (n(y)-1) \rme^{-\rmi k \xhat \cdot y} u(y; \theta) \dy \,, \qquad \xhat \in S^{d-1} \,,
\end{equation}
with $c_d = \rme^{\rmi\pi/4} / \sqrt{8\pi k}$ if $d=2$ and $c_d = 1 / (4\pi)$ if $d=3$ (cf., e.g., \cite[Sec.~6.2]{Kir11}).
Since $\supp(n-1)$ is compact by assumption, the domain of integration in \eqref{eq:farfieldpattern} is actually bounded.

We are interested in analyticity properties of $u_\infty$ as a function on $S^{d-1}\times S^{d-1}$.
To this end, we first recall the definition of a real analytic function on a real analytic manifold (see, e.g., Krantz~\cite{Kra92}):
A subset $\Mcal \subset \R^d$ is called a \emph{manifold} of dimension $m\leq d$ if for each $x\in\Mcal$ there exists a neighborhood $U = U_x$ of $x$ in $\Mcal$, and an open subset $W\subset \R^m$, and a homeomorphism $\phi: U \to \phi(U) = W$.
The pairs $(U, \phi)$ are called \emph{charts}, and an \emph{atlas} is a family $\{ (U_\alpha,\phi_\alpha) \}_{\alpha\in\Acal}$ for some index set $\Acal$ such that $\{ U_\alpha \}_{\alpha\in\Acal}$ is an open covering of $\Mcal$.
The manifold is said to be real analytic, if all transition maps
\begin{equation*}
  \phi_\beta \circ \phi_\alpha^{-1}: \phi_\alpha(U_\alpha\cap U_\beta) \to \phi_\beta(U_\alpha \cap U_\beta) \,, \qquad \alpha, \beta \in \Acal \,,
\end{equation*}
are real analytic.
Accordingly, a mapping $f : \mathcal{M} \to \C$ is called \emph{real analytic} if $f \circ \phi_\alpha^{-1}$ is real analytic for all $\alpha \in \Acal$.

Of course, a real analytic function $f: \Mcal \to \C$ is completely determined by its Taylor coefficients with respect to the local coordinates and therefore by the derivatives 
\begin{equation*}
  \left\{ \big(D^\eta (f \circ \phi^{-1})\big)(y) \; | \; \eta \in \N^m_0 \right\}
\end{equation*}
for an arbitrary $(U,\phi) \in \{(U_\alpha,\phi_\alpha) \}_{\alpha \in \mathcal{A}}$ and $y \in \phi(U)$.

\begin{example}% 
  \label{example-charts-3d}
  The unit sphere $S^{d-1}$, ${d=2,3}$, and therefore also ${S^{d-1}\times S^{d-1}}$, is a real analytic manifold:
  An atlas consisting of two charts is given by 
\begin{equation*}
U_+ = \{ x \in S^{d-1} \;|\; x_d < 3/5 \} \,, \qquad U_- = \{ x \in S^{d-1} \;|\; x_d > -3/5 \} \,,
\end{equation*}
and accordingly
  \begin{equation*}
    \begin{split}
      &\phi_+: U_+ \to \R^{d-1}, \quad (x_1,\ldots, x_d)
      \mapsto \Bigl( \frac{x_1}{1-x_d}, \cdots, \frac{x_{d-1}}{1-x_d} \Bigr) \,,\\
      &\phi_-: U_- \to \R^{d-1}, \quad (x_1,\ldots, x_d)
      \mapsto \Bigl( \frac{x_1}{1+x_d}, \cdots, \frac{x_{d-1}}{1+x_d} \Bigr) \,.
    \end{split}
  \end{equation*}
  It follows immediately that $\phi_+(U_+) = \phi_-(U_-) = B_2(0)$, where $B_2(0) \subset \R^{d-1}$ denotes the open ball of radius $2$ around the origin, and that
\begin{equation*}
    \begin{split}
      \phi_+^{-1}(y) 
      &= \Bigl( \frac{2y_1}{1+\|y\|^2}, \ldots, \frac{2y_{d-1}}{1+\|y\|^2}, \frac{\|y\|^2-1}{\|y\|^2+1} \Bigr) \,, \qquad y = (y_1,\ldots,y_{d-1}) \in \R^{d-1} \,,\\ 
      \phi_-^{-1}(y) 
      &= \Bigl( \frac{2y_1}{1+\|y\|^2}, \ldots, \frac{2y_{d-1}}{1+\|y\|^2}, \frac{1-\|y\|^2}{\|y\|^2+1} \Bigr) \,, \qquad y = (y_1,\ldots,y_{d-1}) \in \R^{d-1} \,.
    \end{split}
  \end{equation*}
Since the transition map
  \begin{equation*}
    (\phi_- \circ \phi_+^{-1})(y) = \frac{y}{\|y\|^2} = (\phi_+ \circ \phi_-^{-1})(y) \,, \qquad y \in B_2(0) \setminus \{0\} \subset \R^{d-1} \,,
  \end{equation*}
 is real analytic and $S^{d-1} = U_+\cup U_-$, we have equipped $S^{d-1}$ with a real analytic structure.

  For later reference we note that $\phi_\pm^{-1}$ canonically extends to a bounded function $\widetilde{\phi_\pm^{-1}}: V \to \C^d$ that is separately holomorphic with respect to each variable $z_j$, $j=1,\ldots,d-1$, with $V \subset \C^{d-1}$ being a sufficiently small open neighborhood of $B_2(0) \subset \R^{d-1}$ in $\C^{d-1}$.
(Here the map $\R^{d-1} \ni y \mapsto \| y \|^2 \in \R$ is extended as $\C^{d-1} \ni z \mapsto \sum z_j^2 \in \C$, not as the squared norm of $\C^{d-1}$.)
  \hfill $\diamond$
\end{example}

From the representation formula \eqref{eq:farfieldpattern} it follows immediately that for each $\theta \in S^{d-1}$ the function $u_\infty(\,\cdot\,;\theta)$ is real analytic on $S^{d-1}$.
By reciprocity,
\begin{equation*}
  u_\infty(\xhat;\theta) = u_\infty(-\theta;-\xhat) \qquad \text{for $\xhat, \theta \in S^{d-1}$} 
\end{equation*}
(see, e.g., \cite[Sec.~6.3]{Kir11}), and thus for each $\xhat \in S^{d-1}$ the function $u_\infty(\xhat;\,\cdot\,)$ is real analytic on $S^{d-1}$ as well.
However, the joint real analyticity of $u_\infty$ with respect to both variables, i.e., that $u_\infty$ locally coincides with its Taylor series in the local coordinates corresponding to $(\hat{x}, \theta)$, is a stronger result: a standard counter example is the function (cf., e.g., \cite[p.~27]{Hormander73})
\begin{equation*}
  f(x,y) = \frac{xy}{x^2 + y^2} \,, \qquad x,y \in \R \,,
\end{equation*}
which is separately real analytic on $\R^2$, but not even continuous at the origin. 
To the best of our knowledge, the joint real analyticity of $u_\infty$ with respect both variables has not been reported in the literature so far.
As our main result, we establish this property of the far field pattern in the following theorem, which will be proved in Section~\ref{sec:proof}. 

\begin{theorem}
  \label{thm:main}
  The far field pattern $u_\infty$ introduced in \eqref{eq:farfieldpattern} is a (jointly) real analytic function on $S^{d-1}\times S^{d-1}$.
\end{theorem}

In particular, Theorem~\ref{thm:main} implies that any uniqueness result for an inverse scattering problem related to \eqref{eq:forward}--\eqref{eq:Sommerfeld} requiring the knowledge of $u_\infty$ everywhere on $S^{d-1}\times S^{d-1}$ remains valid if the values of $u_\infty$ and of all of its derivatives are available at a single point $(\xhat, \theta) \in S^{d-1}\times S^{d-1}$.

\begin{remark}
It follows from a result by F.~E.~Browder~\cite[Thm.~2]{Bro61} that the infinite differentiability of $u_\infty$ on $S^{d-1}\times S^{d-1}$, which is more straightforward to establish than the joint real analyticity, guarantees that $u_\infty$ is real analytic on some open dense subset of $S^{d-1}\times S^{d-1}$. However, this does not ensure that $u_\infty$ is jointly real analytic, say, along the subset $\{(-\theta,\theta) \;|\; \theta \in S^{d-1}\}$ of $S^{d-1}\times S^{d-1}$, which corresponds to the important case of backscattering.

One may also argue that the application of \cite[Thm.~1]{Bro61} would make the proof of the joint analyticity in the following section even shorter. However, we feel that a proof based on Hartogs' theorem and a standard factorization of the far field operator is a more fundamental approach. \hfill $\diamond$
\end{remark}

\section{Proof of the main result}
\label{sec:proof}
One of the main ingredients for the proof of Theorem~\ref{thm:main} is a suitable factorization of the so-called \emph{far field operator} $F: L^2(S^{d-1}) \to L^2(S^{d-1})$,
\begin{equation}
  \label{eq:DefF}
  (F g)(\xhat) := \int_{S^{d-1}} u_\infty(\xhat; \theta) g(\theta)  \ds(\theta) \,,
\end{equation}
where $\ds$ denotes the usual surface measure on $S^{d-1}$. 
Since the kernel of this integral operator is bounded it follows immediately that $F$ is compact.
Furthermore, $F$ can be decomposed as a product of three simpler operators as outlined in the following lemma (see, e.g., \cite{Kirsch02} or \cite[Sec.~4.3]{Kirsch08} for related factorizations).

Before stating this result, we introduce the three operators appearing in the factorization.
To this end let $B\subset \R^d$ be a bounded domain such that ${\supp(n-1) \subset B}$.
Then the \emph{Herglotz operator} $A:L^2(S^{d-1}) \to L^2(B)$ is defined by
\begin{equation}
  \label{eq:DefA}
  (A\psi)(y) := \int_{S^{d-1}} \psi(\theta) \rme^{\rmi k y\cdot\theta} \ds(\theta) \,, \qquad y\in B \,,
\end{equation}
and the corresponding adjoint operator $A^*: L^2(B) \to L^2(S^{d-1})$ 
is given by
\begin{equation}
  \label{eq:DefAstar}
  (A^*\phi)(\xhat) = \int_B \phi(y) \rme^{- \rmi k \xhat\cdot y} \dy \,, \qquad \xhat \in S^{d-1} \,.
\end{equation}
Moreover, given  $f \in L^2(B)$ we consider the source problem
\begin{equation}
\label{eq:Defv}
\Delta v + k^2 n v = -k^2 (n-1) f \qquad \text{in $\R^d$} \,.
\end{equation}
Here and in what follows, we interpret $L^2(B)$ as a subspace of $L^2(\R^d)$ by identifying the elements of $L^2(B)$ with their zero continuations. As in \cite[Sec.~6.2]{Kir11}, it can be seen that \eqref{eq:Defv} together with the Sommerfeld radiation condition \eqref{eq:Sommerfeld} has a unique weak solution $v \in H^1_\loc(\R^d)$. Moreover, it easily follows that the linear operator $T: L^2(B) \to L^2(B)$, 
\begin{equation}
\label{eq:T}
Tf := k^2 (n-1) (f+v|_B) \,,
\end{equation}
is bounded and the far field pattern ${v_\infty \in L^2(S^{d-1})}$ of $v$ depends continuously on $f \in L^2(B)$.

\begin{lemma}
  \label{lmm:factorization}
  The far field operator $F$ can be factored as
  \begin{equation*}
    F = c_d A^* T A
  \end{equation*}
  with $A$, $A^*$, and $T$ as in \eqref{eq:DefA}, \eqref{eq:DefAstar}, and \eqref{eq:T}, respectively.
\end{lemma}

\begin{proof}
To begin with, we define a bounded operator $W: L^2(B) \to L^2(S^{d-1})$,
\begin{equation*}
 Wf := v_\infty \,,
\end{equation*}
where $v_\infty$ is the far field pattern of the solution to \eqref{eq:Defv} with \eqref{eq:Sommerfeld}.
Observe that any solution $u = u^i + u^s$ of \eqref{eq:forward} satisfies
  \begin{equation*}
    \Delta u^s(\,\cdot\,;\theta) + k^2 n u^s(\,\cdot\,;\theta) = -k^2 (n-1) u^i(\,\cdot\,;\theta) \qquad \text{in $\R^d$},
  \end{equation*}
and thus it follows via superposition from the definition of $F$ in \eqref{eq:DefF} and of $A$ in \eqref{eq:DefA} that $F = WA$.

 For any $\phi \in L^2(B) $, it holds that $A^*\phi = c_d^{-1} w_\infty$, where $w \in H^1_\loc(\R^d)$ is the unique weak solution of
  \begin{equation*}
    \Delta w + k^2 w = -\phi \qquad \text{in $\R^d$}
  \end{equation*}
  coupled with \eqref{eq:Sommerfeld}.
 In particular, if $\phi = Tf \in L^2(B)$ for some $f \in L^2(B)$, then
  \begin{equation*}
    \Delta w + k^2 w = -Tf =  -k^2(n-1)(f+v|_B) = \Delta v + k^2 v \qquad \text{in $\R^d$} \,,
  \end{equation*}
where $v$ is still the solution of \eqref{eq:Defv} with \eqref{eq:Sommerfeld}. 
Consequently, $v=w$ in $\R^d$, and we conclude that $A^*Tf = c_d^{-1} v_\infty = c_d^{-1} Wf$ for any $f \in L^2(B)$. Altogether we have established that $F = WA = c_d A^*TA$, which completes the proof.
\end{proof}

Since $A$ is an integral operator with a smooth kernel, it can be extended to a bounded operator from $H^{-s}(S^{d-1})$ to $L^2(B)$ for any $s \geq 0$.
Similarly, $A^*$ is bounded from $L^2(B)$ to $H^s(S^{d-1})$, $s \geq 0$.
Due to the denseness of $L^2(S^{d-1})$ in $H^{-s}(S^{d-1})$ for any $s \geq 0$
(cf., e.g., \cite{Lions72}), $A^*$ remains the adjoint of the extended $A$ in the sense that
$$
\int_B (A \psi)  \overline{\phi} \dy = \big \langle \psi, \overline{A^* \phi} \big\rangle_{S^{d-1}},\qquad \psi \in H^{-s}(S^{d-1}), \ \phi \in L^2(B),
$$
where $\langle \cdot,\cdot \rangle_{S^{d-1}}$ is the (bilinear) dual evaluation between $H^{-s}(S^{d-1})$ and $H^{s}(S^{d-1})$.
In particular, $F$ extends to an operator from $\mathscr{D}'(S^{d-1})$ to $\mathscr{D}(S^{d-1})$ that is bounded from $H^{-s}(S^{d-1})$ to $H^{s}(S^{d-1})$ for any $s \in \R$. 

\begin{corollary}
  \label{corollary:ff:pattern:bilinear:form}
  The far field pattern $u_\infty$ can be represented as
  \begin{equation*}
    u_\infty(\hat{x};\theta) =  
    c_d \int_B \rme^{-{\rm i} k \hat{x} \cdot y} \bigl(T( \rme^{{\rm i}k \theta \cdot \cdot})\bigr)(y) \dy
  \end{equation*}
  for all $(\hat{x};\theta) \in S^{d-1}\times S^{d-1}$.
\end{corollary}

\begin{proof}
  Denoting by
  $\delta_{\hat{z}} \in H^{-(d-1)/2-\varepsilon}(S^{d-1})$, $\varepsilon > 0$,  
  the delta distribution with singularity at $\hat{z}$ on $S^{d-1}$, we may write
  \begin{equation*}
    u_\infty(\hat{x};\theta) 
    = \overline{ \langle \delta_{\hat{x}}, \overline{F \delta_{\theta}} \rangle}_{S^{d-1}}
    = c_d \overline{\int_B (A \delta_{\hat{x}}) (\overline{TA\delta_{\theta}}) \dy}
    = c_d \int_{B} \rme^{-{\rm i} k \hat{x} \cdot y} \bigl(T( \rme^{{\rm i}k \theta \cdot \cdot})\bigr)(y) \dy \,,
  \end{equation*}
and the proof is complete.
\end{proof}

Now, let $(U,\phi)$ be a  chart in the analytic atlas of $S^{d-1}$ as described in Example~\ref{example-charts-3d}.
We define $g : \phi(U) \times \R^d \to \C$ by
\begin{equation*}
  g(x,y) = \rme^{-{\rm i} k \phi^{-1}(x) \cdot y} \,.
\end{equation*} 
Using the holomorphic extension $\widetilde{\phi^{-1}} : V \rightarrow \C^d$ of $\phi^{-1}$ from Example~\ref{example-charts-3d}, we immediately obtain the corresponding extension $\tilde g : V \times \R^d \rightarrow \C$,
\begin{equation*}
  \tilde g(z,y) = \rme^{-{\rm i} k \widetilde{\phi^{-1}}(z) \cdot y} \,
\end{equation*} 
of $g$. 
Obviously, $\tilde{g}(z,y)$ is holomorphic with respect to $z \in V$ for any $y \in \R^d$.

\begin{lemma}
  \label{lemma:G:holomorphic:each}
  Let $G : V \rightarrow L^2(B)$ be given by $G(z) = \tilde{g}(z,\cdot)$.
  Then $G$ is holomorphic separately in each variable $z_j$, $j=1,\ldots,d-1$, when the other variables assume arbitrary fixed values.
\end{lemma}

\begin{proof}
  Pick an arbitrary $1\leq j \leq d-1$ and let $M > 0$ be such that
  \begin{equation*}
    \Bigl| \frac{\partial^2}{\partial z_j^{2}} \tilde g(z,y) \Bigr| \leq M
  \end{equation*}
  for all $z = (z_1,\ldots,z_{d-1})$ in the bounded set $V \subset \C^{d-1}$ and $y \in B$.

  Fix $z \in V$ and let $r > 0$ be such that $\{ w \in \C^{d-1} \, : \, | w - z | < r\}\subset V$. 
  Denoting for any ${\eta = (0,\ldots,0,\eta_j,0,\ldots,0) \in \C^{d-1}}$ by $\Gamma \subset \C^{d-1}$ the straight line between $z$ and $z+\eta$, we find that
  \begin{equation*}
      \Bigl|\frac{\tilde g(z + \eta,y) - \tilde g(z,y)}{\eta_j} - \frac{\partial \tilde g}{\partial z_j}(z,y) \Bigr|
      \leq \Bigl| \frac1{\eta_j} \int_\Gamma \Bigl( \frac{\partial \tilde g}{\partial z_j}(x,y) - \frac{\partial \tilde g}{\partial z_j}(z,y)\Bigr) \dx \Bigr|
      \leq \frac{1}{2}M |\eta_j|
  \end{equation*}
  for all $y \in B$ and $0 \not = \eta_j \in \C$ such that $| \eta_j | < r$. 
  Thus,
  \begin{align*}
    \Bigl\| \frac{ G(z + \eta) - G(z)}{\eta_j} - \frac{\partial G}{\partial z_j}(z) \Bigr\|_{L^2(B)}
    \leq \frac12 M |\eta_j| \sqrt{|B|} \,,
  \end{align*}
  which means that $G$ is holomorphic in $z_j$. 
\end{proof}

Finally, we provide a proof for our main result.

\begin{proof}[Proof of Theorem~\ref{thm:main}]
  Suppose $\hat x, \theta \in S^{d-1}$ and let $(U_{\hat{x}},\phi_{\hat{x}})$, $(U_{\theta},\phi_\theta)$ be charts as in Example~\ref{example-charts-3d} such that $\hat{x}\in U_{\hat{x}}$ and $\theta \in U_{\theta}$.
  Considering the holomorphic extensions $\widetilde{\phi^{-1}_{\hat{x}}}, \widetilde{\phi^{-1}_\theta} : V \rightarrow \C^d$, we define $G_{\hat{x}}, G_\theta' : V  \rightarrow L^2(B)$ by
  \begin{equation*}
    [G_{\hat{x}}(z)](y) = \rme^{-{\rm i} k \widetilde{\phi_{\hat{x}}^{-1}}(z) \cdot y}, \qquad [G_\theta'(w)](y) = ([G_\theta(w)](y))^{-1} = \rme^{{\rm i} k \widetilde{\phi_\theta^{-1}}(w) \cdot y}.
  \end{equation*}
  Due to Lemma~\ref{lemma:G:holomorphic:each}, $G_{\hat{x}}$ and $G'_\theta$ are holomorphic separately in each variable when the other variables assume arbitrary fixed values. 

  As $T: L^2(B) \to L^2(B)$ is continuous, the bilinear form 
  \begin{equation*}
    (p,q) := \int_B p \; Tq \dy
  \end{equation*}
  is bounded from $L^2(B) \times L^2(B)$ to $\C$, and it easily follows that the (local) extension
  \begin{equation*}
    u_\infty \bigl(\widetilde{\phi_{\hat{x}}^{-1}}(z), \widetilde{\phi_\theta^{-1}}(w)\bigr) 
    = c_d \int_{B}  \bigl(G_{\hat{x}}(z)\bigr)(y) \bigl(T (G_\theta'(w))\bigr)(y) \dy 
  \end{equation*} 
  of the far field pattern is holomorphic in $V\times V$ with respect to each of the coordinates separately. 
  By Hartogs' theorem (see, e.g., \cite[Thm.~1.2.5]{Kra92} or \cite[Thm.~2.2.8]{Hormander73}), it is therefore a (jointly) holomorphic function in $V\times V$, and thus it is also (jointly) real analytic on $(V\times V) \cap (\R^{d-1} \times \R^{d-1})$ (cf.~\cite[Cor.~2.3.7]{Kra92}).

  Consequently, the local representation 
  \begin{equation*}
    (z,w) \mapsto u_\infty(\phi_{\hat{x}}^{-1}(z), \phi_\theta^{-1}(w)) 
  \end{equation*} 
  of the far field pattern is a (jointly) real analytic function on $B_2(0)\times B_2(0)\subset \R^{d-1}\times \R^{d-1}$. 
  Since $\hat{x}, \theta \in S^{d-1}$ were arbitrary, this shows that $u_\infty$ is (jointly) real analytic on ${S^{d-1}\times S^{d-1}}$. 
\end{proof}

\section*{Acknowledgements}
Part of this research was carried out during a visit of RG and NH at the Department of Mathematical Sciences at the University of Delaware.
RG and NH would like to thank Prof.~F.~Cakoni, Prof.~D.~Colton and Prof.~P.~Monk for the kind invitation and their hospitality.


\begin{thebibliography}{99}

  \bibitem{Bro61}
  F.~E.~Browder, 
  {\em Real analytic functions on product spaces and separate analyticity}, 
  Canad. J. Math. 13 (1961),  650--656. 

\bibitem{Buk08}
  A.~L.~Bukhgeim,
  {\em Recovering a potential from Cauchy data in the two-dimensional case},
  J.\ Inverse Ill-Posed Probl., 16 (2008), 19--33.

\bibitem{ColKre83}
  D.~L.~Colton and R.~Kress, 
  {\em Integral equation methods in scattering theory},
  John Wiley \& Sons, New York, 1983.

\bibitem{Colton92}
  D. Colton and R. Kress, 
  {\em Inverse Acoustic and Electromagnetic Scattering Theory}, 2.~ed., 
  Springer, Berlin, 1998.

\bibitem{Har06}
  F.~Hartogs, 
  {\em Zur Theorie der analytischen Funktionen mehrerer unabhängiger Veränderlichen, insbesondere über die Darstellung derselben durch Reihen, welche nach Potenzen einer Veränderlichen fortschreiten},
  Math.\ Ann., 62 (1906), 1--88.

\bibitem{Hormander73}
  L.~H\"ormander,
  {\em An Introduction to Complex Analysis in Several Variables}, 3rd~ed., North-Holland, Amsterdam, 1990. 

\bibitem{Hyvonen12}
  N.~Hyv\"onen, P.~Piiroinen and O.~Seiskari,
  {\em Point measurements for a Neumann-to-Dirichlet map and the Calder\'on problem in the plane}, SIAM J. Math. Anal., accepted.

\bibitem{Kirsch02}
  A. Kirsch, 
  {\em The MUSIC algorithm and the factorization method in inverse scattering theory for inhomogeneous media},
  Inverse Problems, 18 (2002), 1025--1040.

\bibitem{Kirsch08}
  A. Kirsch,
  {\em The Factorization Method for Inverse Problems}, Oxford University Press,
  Oxford, 2008.

\bibitem{Kir11}
  A. Kirsch,
  {\em An introduction to the mathematical theory of inverse problems}, 2nd~ed.,
  Springer, New York, 2011.

\bibitem{Kra92}
  S.~G.~Krantz,
  {\em Function theory of several complex variables}, 2nd~ed., Wadsworth \& Brooks/Cole Advanced Books \& Software, Pacific Grove, CA, 1992.

   \bibitem{Lions72} 
     J.-L.~Lions and E.~Magenes,
     {\em Non-Homogeneous Boundary Value Problems and Applications}, Vol.~I,
    Springer, Berlin, 1972.

\bibitem{Nac88}
  A.~I.~Nachman,
  {\em Reconstructions from boundary measurements},
  Ann.\ of Math.\, 128 (1988), 531--576.

\bibitem{Nov88}
  R.~G.~Novikov,
  {\em A multidimensional inverse spectral problem for the equation $-\Delta \psi + (v(x)-Eu(x))\psi=0$}, 
  translation in Funct.\ Anal.\ Appl.\, 22 (1988), 263--272.

\bibitem{Ram88}
  A.~G.~Ramm,
  {\em Recovery of the potential from fixed-energy scattering data},
  Inverse Problems, 4 (1988), 877--886.

\end{thebibliography}
\end{document}